\documentclass[12pt]{amsart}

\usepackage{amssymb}
\usepackage{bm}
\usepackage{graphicx}
\usepackage{psfrag}
\usepackage{color}
\usepackage{hyperref}
\hypersetup{colorlinks=true, linkcolor=blue, citecolor=magenta, urlcolor=wine}
\usepackage{url}
\usepackage{algpseudocode}
\usepackage{tikz-cd}
\usepackage{fancyhdr}
\usepackage{xy}
\input xy
\xyoption{all}
\usepackage{stmaryrd}
\usepackage{calrsfs}

\newtheorem*{maintheorem*}{Main Theorem}
\newtheorem{theorem}{Theorem}[section]

\newtheorem{question}[theorem]{Question}
\newtheorem{prop}[theorem]{Proposition}

\newtheorem{cor}[theorem]{Corollary}

\theoremstyle{definition}
\newtheorem{definition}[theorem]{Definition}
\newtheorem{remark}[theorem]{Remark}
\newtheorem{example}[theorem]{Example}
\numberwithin{equation}{section}

\newcommand{\nn}{\mathbb{N}}
\newcommand{\pp}{\mathbb{P}}
\newcommand{\qq}{\mathbb{Q}}
\newcommand{\rr}{\mathbb{R}}
\newcommand{\zz}{\mathbb{Z}}

\providecommand\ldb{\llbracket}
\providecommand\rdb{\rrbracket}

\newcommand{\gp}{\textsf{gp}}
\newcommand{\rank}{\text{rank}}

\newcommand{\supp}{\text{supp}}

\keywords{positive monoids, Puiseux monoids, atomicity, factorization theory, ACCP, BF-monoids, FF-monoids}

\subjclass[2010]{Primary: 20M13; Secondary: 06F05, 20M14}

\begin{document}
	
	\mbox{}
	\title{Atomicity of positive monoids}
	
	\author{Scott T. Chapman}
	\address{Department of Mathematics\\Sam Houston State University\\Huntsville, TX 77341}
	\email{scott.chapman@shsu.edu}
	
	\author{Marly Gotti}
	\address{Research and Development \\ Biogen \\ Cambridge \\ MA 02142}
	\email{marly.cormar@biogen.com}
	
	\date{\today}
	
	\begin{abstract}
		An additive submonoid of the nonnegative cone of the real line is called a positive monoid. Positive monoids consisting of rational numbers (also known as Puiseux monoids) have been the subject of several recent papers.  Moreover, those generated by a geometric sequence have also received a great deal of recent attention. Our purpose is to survey many of the recent advances regarding positive monoids, and we provide numerous examples to illustrate the complexity of their atomic and arithmetic structures.
	\end{abstract}
	\bigskip

\maketitle

\bigskip


\section{Introduction}
\label{sec:intro}

A cancellative and commutative (additive) monoid is called \textit{atomic} if every non-invertible element is the sum of \textit{atoms} (i.e., \textit{irreducibles}). Much recent literature has focused on the arithmetic of such monoids; the monograph \cite{GH06} contains an extensive bibliography of such work. Many of these recent works have centered on important classes of monoids such as Krull monoids, the multiplicative monoids of integral domains, numerical monoids, and congruence monoids.  In their landmark study of the multiplicative monoid of an integral domain \cite{AAZ90}, Anderson, Anderson, and Zafrullah introduced the properties of bounded and finite factorizations. These ideas can easily be extended to all commutative cancellative monoids, and we include below a diagram~\eqref{diag:AAZ's chain for monoids} containing their factorization properties modified for the general case.
\begin{equation} \label{diag:AAZ's chain for monoids}
	\begin{tikzcd}
		\textbf{ UFM } \ \arrow[r, Rightarrow]  \arrow[d, Rightarrow] & \ \textbf{ HFM } \arrow[d, Rightarrow] \\
		\textbf{ FFM } \ \arrow[r, Rightarrow] & \ \textbf{ BFM } \arrow[r, Rightarrow]  & \textbf{ ACCP monoid}  \arrow[r, Rightarrow] & \textbf{ atomic monoid}
	\end{tikzcd}
\end{equation}
While it is well known in the general case that each implication in ~\eqref{diag:AAZ's chain for monoids} holds, it is also known that none of the implications are reversible (even in the class of integral domains, see~\cite{AAZ90}).
\smallskip

The fundamental purpose of this work is to survey the recent results regarding the atomicity of additive submonoids of the nonnegative cone of the real line.  Such a survey
 is important, as in many cases these monoids offer simpler examples of complex factorization properties, than those currently in the commutative algebra literature.   We begin with the following definition.

\begin{definition}
	An additive submonoid of $(\rr_{\ge 0}, +)$ is called a \emph{positive monoid}.
\end{definition}


Submonoids of $(\nn_0,+)$ are clearly positive monoids, and they are called \emph{numerical monoids}. An introduction to numerical monoids is offered in~\cite{GR09}. In addition, submonoids of $(\qq_{\ge 0},+)$ are called \emph{Puiseux monoids}. Every numerical monoid is clearly a Puiseux monoid, and one can readily verify that a Puiseux monoid is finitely generated if and only if it is isomorphic to a numerical monoid. Puiseux monoids are perhaps the positive  monoids that have been most systematically  investigated in the last five years (see \cite{GGT21} and references therein). A survey on the atomicity of Puiseux monoids can be found in the recent Monthly article~\cite{CGG21}.

Positive monoids that are increasingly generated have been studied in~\cite{mBA20,BG21,fG19,GG18}. On the other hand, positive semirings (i.e., positive monoids closed under multiplication) have been studied in \cite{BCG21,BG20}, while the special case of positive monoids generated by a geometric sequence (necessarily positive semirings) have been studied in~\cite{CGG20,CG20}. Finally, Furstenberg positive monoids (i.e., those whose nonzero elements are divisible by an atom) have been recently investigated in~\cite{fG21}.


We break our remaining work into 5 sections.  In Section 2, we lay out the basic necessary definitions and notation.  Section 3 explores the factorization properties of cyclic semirings (i.e., additive monoids generated over the nonnegative integers by the sequence $\{\alpha^n\}_{n\in \mathbb{N}_0}$ where $\alpha$ is a positive real number).  In Theorem 3.3, we characterize when these monoids are atomic, and in those cases determine completely in Proposition 3.6 which elements are atoms.  In Section 4, we explore constructing atomic positive monoids that do not satisfy the ACCP.  These examples are vital, as such examples in the realm of integral domains are extremely difficult to construct.  As Proposition 4.4 shows us how to do this with arbitrary rank, Proposition 4.6 constructs a positive monoid with the ACCP which is not a BFM.  Section 5 explores bounded factorizations, and Proposition 5.5 constructs positive monoids which are BFMs.  As a by product of these results, we offer several examples of BFMs which are not FFMs.  We conclude in Section 6 by exploring in detail the finite factorization property; we prove in Theorem 6.1 that every positive monoid generated by an increasing sequence is an FFM.  In some sense, our entire paper is motivated by diagram \eqref{diag:AAZ's chain for monoids}.  We offer counterexamples using positive monoids to all the reverse implications of \eqref{diag:AAZ's chain for monoids} (see Remarks 3.1, 4.5, 4.7, 5.4, 6.5, and 6.7).

\bigskip 


\section{Preliminaries}
\label{sec:prelim}

We let $\pp$, $\nn$, and $\nn_0$ denote the set of primes, positive integers, and nonnegative integers, respectively. If $X$ is a subset of $\rr$ and $r$ is a real number, we let $X_{\ge r}$ denote the set $\{s \in X : s \ge r\}$. In a similar fashion, we use the notations $X_{> r}, X_{\le r}$, and $X_{< r}$. For a positive rational $q$, the positive integers $a$ and $b$ with $q = a/b$ and $\gcd(a,b) = 1$ are denoted by $\mathsf{n}(q)$ and $\mathsf{d}(q)$, respectively. 
\smallskip

The following definition of a monoid, albeit not the most standard\footnote{A monoid is most commonly defined as a semigroup with an identity element.}, will be the most appropriate in the context of this paper.

\begin{definition}
	A \emph{monoid} is a semigroup with identity that is cancellative and commutative.
\end{definition}

Monoids will be written additively, unless we say otherwise. In addition, we shall tacitly assume that every monoid here is reduced, that is, its only invertible element is zero. Let $M$ be a monoid. We set $M^\bullet = M \setminus \{0\}$. For a subset $S$ of $M$, we let $\langle S \rangle$ denote the submonoid of $M$ generated by $S$, i.e., the intersection of all submonoids of $M$ containing $S$. We say that a monoid is \emph{finitely generated} if it can be generated by a finite set.
\smallskip

A nonzero element $a \in M$ is called an \emph{atom} if whenever $a = b_1 + b_2$ for some $b_1, b_2 \in M$ either $b_1 = 0$ or $b_2 = 0$. As it is customary, we let $\mathcal{A}(M)$ denote the set consisting of all atoms of $M$. If $\mathcal{A}(M)$ is empty, $M$ is said to be \emph{antimatter}.

\begin{definition}
	A monoid is \emph{atomic} if every nonzero element of the monoid is the sum of atoms.
\end{definition}

If $I$ is a subset of $M$, then $I$ is called an \emph{ideal} provided that $I + M = I$. Ideals of the form $b + M$, where $b \in M$, are called \emph{principal}. The monoid $M$ satisfies the \emph{ascending chain condition on principal ideals} (\emph{ACCP} for short) if every ascending chain of principal ideals of $M$ becomes stationary from some point onward. It is not hard to prove that every monoid satisfying the ACCP is atomic (see \cite[Proposition~1.1.4]{GH06}).
\smallskip

A  monoid $F$ is called a \emph{free commutative monoid} with basis~$A$ if every element $b \in F$ can be written uniquely as the sum of elements in $A$. It is well known that for every set $A$ there exists, up to isomorphism, a unique free commutative monoid on $A$, which we denote by $F(A)$. It is also well known that every map $A \to M$, where $M$ is a monoid, uniquely extends to a monoid homomorphism $F(A) \to M$.
\smallskip

The \emph{Grothendieck group} of a monoid $M$, here denoted by $\gp(M)$, is the abelian group (unique up to isomorphism) satisfying the property that any abelian group containing a homomorphic image of $M$ will also contain a homomorphic image of $\gp(M)$. The \emph{rank} of the monoid $M$ is then defined to be the rank of $\gp(M)$ as a $\zz$-module or, equivalently, the dimension of the $\qq$-vector space $\qq \otimes_\zz \gp(M)$. 
\smallskip

For an atomic monoid $M$, we let $\mathsf{Z}(M)$ denote the free commutative monoid on the set $\mathcal{A}(M)$. The elements of $\mathsf{Z}(M)$ are called \emph{factorizations}. Then we can think of factorizations in $\mathsf{Z}(M)$ as formal sums of atoms. The unique monoid homomorphism $\pi \colon \mathsf{Z}(M) \to M$ such that $\pi(a) = a$ for all $a \in \mathcal{A}(M)$ is called the \emph{factorization homomorphism}. For every element $b \in M$,
\[
	\mathsf{Z}(b) := \pi^{-1}(b) \subseteq \mathsf{Z}(M)
\]
is called the \emph{set of factorizations} of $b$. If for every $b \in M$ the set $\mathsf{Z}(b)$ is finite, then $M$ is called a \emph{finite factorization monoid} (\emph{FFM}). Also, if $\mathsf{Z}(b)$ is a singleton for every $b \in M$, then $M$ is called a \emph{unique factorization monoid} (\emph{UFM}). Note that every UFM is an FFM. It follows from \cite[Proposition~2.7.8]{GH06} that every finitely generated monoid is an FFM.
\smallskip

Let $z \in \mathsf{Z}(M)$ be a factorization in $M$. If we let $|z|$ denote the number of atoms (counting repetitions) in the formal sum defining $z$ in $\mathsf{Z}(M)$, then $|z|$ is called the \emph{length} of $z$. For each element $b \in M$,
\[
	\mathsf{L}(b) := \{|z| \ :  z \in \mathsf{Z}(b)\}
\]
is called the \emph{set of lengths} of $b$. If the set $\mathsf{L}(b)$ is finite for each $b \in M$, then $M$ is called a \emph{bounded factorization monoid} (\emph{BFM}).  It is clear that FFMs are BFMs. The finite and bounded factorization properties were introduced by Anderson, Anderson, and Zafrullah in \cite{AAZ90} in the context of integral domains. Bounded and finite factorization monoids were first studied by Halter-Koch in~\cite{fHK92}. A recent survey on the finite and bounded factorization properties can be found in~\cite{AG20}. The monoid $M$ is called a \emph{half-factorial monoid} (\emph{HFM}) provided that for every $b\in M$ the set $\mathsf{L}(b)$ is a singleton. It follows directly from the definition that every HFM is a BFM. The study of half-factoriality, mainly in the context of algebraic number theory, dates back to the 1960s (see \cite{lC60}). The term ``half-factorial" was coined by Zaks in \cite{aZ76}. A survey on half-factoriality can be found in~\cite{CC00}.

\bigskip


\section{A Class of Atomic Positive Monoids}
\label{sec:atomicity}

A positive monoid consisting of rational numbers is called a \emph{Puiseux monoid}. The class of Puiseux monoids will be a convenient source of examples throughout our exposition. None of the implications of Diagram~\ref{diag:AAZ's chain for monoids} are reversible in the class of positive monoids.  Moreover, as is illustrated in~\cite{CGG21}, none of the implications (except \textbf{UFM} $\Rightarrow$ \textbf{HFM}) is reversible in the subclass of Puiseux monoids. An example of a half-factorial positive monoid that is not a UFM is given in Example~\ref{ex:HFM not UFM}.

\begin{remark}
	In this section, we primarily focus on atomic monoids. However, there are many positive monoids that are not atomic. Indeed, the Puiseux monoid $\langle 1/p^n : n \in \nn \rangle$ is antimatter for every $p \in \pp$. 
\end{remark}

Perhaps the class of non-finitely generated positive monoids that has been most thoroughly studied is that one consisting of cyclic semirings~\cite{CGG20}.

\begin{definition}
	For $\alpha \in \rr_{> 0}$, we let $\nn_0[\alpha]$ denote the positive monoid $\langle \alpha^n : n \in \nn_0 \rangle$.
\end{definition}

Observe that $\nn_0[\alpha]$ is closed under multiplication and, therefore, $(\nn_0[\alpha]^\bullet, \cdot)$ is also a monoid (not necessarily reduced). Positive monoids closed under multiplication are called \emph{positive semirings} and have been recently studied in \cite{BCG21} by Baeth, Gotti, and the first author. We will only be concerned here with the additive structure of the semiring $\nn_0[\alpha]$. For $q \in \qq_{> 0}$, the atomicity of $\nn_0[q]$ was first considered in \cite[Section~6]{GG18}; later in~\cite{CGG20} several factorization invariants of $\nn_0[q]$ were compared and contrasted to those of numerical monoids generated by arithmetic sequences.

\smallskip
In the next theorem, we characterize when $\nn_0[\alpha]$ is atomic. In addition, we give two sufficient conditions for atomicity. First, we recall Descartes' Rule of Signs. Given a polynomial $f(x) = c_n x^n + \cdots + c_1x + c_0 \in \rr[x]$, the \emph{number of variations of the sign} of $f(x)$ is the cardinality of the set $\{j \in \ldb 1, n \rdb : c_j c_{j-1} < 0\}$. Descartes' Rule of Signs states that the number of variations of the sign of a polynomial $f(x) \in \rr[x]$ is at least, and has the same parity as, the number of positive roots of $f(x)$ provided that we count each root with multiplicity.

\begin{theorem}\cite[Theorem~4.1]{CG20} \label{thm:atomic characterization}
	For every $\alpha \in \rr_{>0}$, the following conditions are equivalent.
	\begin{enumerate}
		\item[(a)] The monoid $\nn_0[\alpha]$ is atomic.
		\smallskip
		
		\item[(b)] The monoid $\nn_0[\alpha]$ is not antimatter.
		\smallskip
		
		\item[(c)] The element $1$ is an atom of  $\nn_0[\alpha]$.
	\end{enumerate}
	In addition, if $\alpha$ is an algebraic number and $m(x)$ is the minimal polynomial of $\alpha$, then the following statements hold.
		\begin{enumerate}
			\item If $\alpha$ is not rational and $|m(0)| \neq 1$, then $\nn_0[\alpha]$ is atomic.
			\smallskip
			
			\item If $m(x)$ has at least two positive roots, counting repetitions, then $\nn_0[\alpha]$ is atomic.
		\end{enumerate}
\end{theorem}

\begin{proof}
	(a) $\Rightarrow$ (b): This is clear.
	\smallskip
	
	(b) $\Rightarrow$ (c): Suppose that $1 \notin \mathcal{A}(\nn_0[\alpha])$. Then $1 = \sum_{i=1}^k c_i \alpha^i$ for some $c_1, \dots, c_k \in \nn_0$ with $\sum_{i=1}^k c_i \ge 2$. As a result, $\alpha^n = \sum_{i=1}^k c_i \alpha^{i+n}$ for all $n \in \nn_0$. This implies that $\nn_0[\alpha]$ is antimatter.
	\smallskip
	
	(c) $\Rightarrow$ (a): If $\alpha \ge 1$, then for each $n \in \nn$ the set $\nn_0[\alpha] \cap [0,n]$ is finite, and so the elements of $\nn_0[\alpha]$ are the terms of an increasing sequence. Therefore $\nn_0[\alpha]$ is atomic by~\cite[Theorem~5.6]{fG19}. Now suppose that $\alpha \in (0,1)$. As $\alpha < 1$, we see that $\alpha^i \nmid_{\nn_0[\alpha]} \alpha^j$ whenever $i < j$. Because $1 \in \mathcal{A}(\nn_0[\alpha])$, it follows that $\alpha^n \in \mathcal{A}(\nn_0[\alpha])$ for all $n \in \nn_0$. Thus, $\nn_0[\alpha]$ is atomic, as desired.
	\smallskip
	
	Now assume that $\nn_0[\alpha]$ is atomic, and let us proceed to argue (1) and (2).
	\smallskip
	
	(1) Suppose, for the sake of a contradiction, that  the monoid $\nn_0[\alpha]$ is not atomic. By Theorem~\ref{thm:atomic characterization}, $1$ is not an atom of $\nn_0[\alpha]$ and, therefore, there exist $c_1, \dots, c_n \in \nn_0$ with $1 = \sum_{i=1}^n c_i \alpha^i$. Hence $\alpha$ is a root of $p(x) := 1 - \sum_{i=1}^n c_ix^i \in \qq[x]$. Then write $p(x) = m(x)q(x)$ for some $q(x) \in \qq[x]$. Observe that Gauss' Lemma guarantees that $q(x)$ belongs to $\zz[x]$. Since $p(0) = 1$, the equality $|m(0)| = 1$ holds, which gives the desired contradiction.
	\smallskip
	
	(2) Assume, by way of contradiction, that the monoid $\nn_0[\alpha]$ is not atomic, and write $1 = \sum_{i=1}^n c_i \alpha^i$ for some $c_1, \dots, c_n \in \nn_0$. As we have seen in the previous paragraph, $\alpha$ is a root of the polynomial $p(x) = 1 - \sum_{i=1}^n c_ix^i \in \qq[x]$. It follows now from Descartes' Rule of Signs that $\alpha$ is the only positive root which $p(x)$ can have. Since $m(x)$ is a divisor of $p(x)$ in $\qq[x]$, each root of $m(x)$ must be a root of $p(x)$. Hence $\alpha$ is the only positive root of $m(x)$, a contradiction.
\end{proof}

 The condition $1 \in \mathcal{A}(\nn_0[\alpha])$ in part~(c) of Theorem~\ref{thm:atomic characterization} does not hold in general, which means that there are algebraic numbers $\alpha$ giving antimatter monoids $\nn_0[\alpha]$.

\begin{example}
	Take $\alpha = \frac{\sqrt{5} - 1}{2}$, whose minimal polynomial is $m(x) = x^2 + x - 1$. Since $\alpha$ is a root of $m(x)$, we see that $1 = \alpha^2 + \alpha$. As a consequence, $1$ is not an atom of $\nn_0[\alpha]$, and so Theorem~\ref{thm:atomic characterization} guarantees that $\nn_0[\alpha]$ is antimatter.
\end{example}

As the next example illustrates, none of the sufficient conditions for atomicity we gave as part of Theorem~\ref{thm:atomic characterization} implies the other.

\begin{example}
	Consider the polynomial $m_1(x) = x^2 - 4x + 1$. It is clearly irreducible, and it has two distinct positive real roots: $2 \pm \sqrt{3}$. However, we see that $|m_1(0)| = 1$. Now consider the polynomial $m_2(x) = x^2 + 2x - 2$. It is also irreducible, and it has only one positive real root, namely, $\alpha = \sqrt{3} - 1$. However, $|m_2(0)| \neq 1$.
\end{example}

\smallskip
We proceed to describe the set of atoms of $\nn_0[\alpha]$ when it is atomic.

\begin{prop} \cite[Theorem~4.1]{CG20} \label{prop:atoms of cyclic algebraic semirings}
	If $\nn_0[\alpha]$ is atomic, then the following statements hold.
	\begin{enumerate}
		\item If $\alpha$ is transcendental, then $\mathcal{A}(\nn_0[\alpha]) = \{\alpha^n : n \in \nn_0\}$.
		\smallskip
		
		\item If $\alpha$ is algebraic and $\sigma :=  \min \{n \in \nn \cup \{\infty\} : \alpha^n \in \langle \alpha^j : j \in \ldb 0,n-1 \rdb \rangle \}$, then 
		\begin{itemize}
			\item if $\sigma < \infty$, then $\mathcal{A}(\nn_0[\alpha]) = \{\alpha^n : n \in \ldb 0, \sigma-1 \rdb\}$, and
			\smallskip
			
			\item if $\sigma = \infty$, then $\mathcal{A}(\nn_0[\alpha]) = \{ \alpha^n : n \in \nn_0 \}$.
		\end{itemize}
	\end{enumerate}
\end{prop}

\begin{proof}
	(1) Suppose that $\alpha$ is transcendental. If $\alpha^n = \sum_{i=0}^N c_i \alpha^i$ for some $n \in \nn_0$, $N \in \nn_{\ge n}$, and coefficients $c_0, \dots, c_N \in \nn_0$, then $\alpha$ is a root of the polynomial $x^n - \sum_{i=0}^N c_i x^i \in \qq[x]$. Since $\alpha$ is transcendental, $c_i = 0$ for every $i \in \ldb 0,N \rdb \setminus \{n\}$ and $c_n = 1$, which implies that $\alpha^n$ is an atom. Hence $\mathcal{A}(\nn_0[\alpha]) = \{\alpha^n : n \in \nn_0\}$, as desired.
	\smallskip
	
	(2) Now suppose that $\alpha$ is an algebraic number. First, we will assume that $\sigma \in \nn$. Since $\alpha^{\sigma} \in \langle \alpha^n : n \in \ldb 0, \sigma-1 \rdb \rangle$, it follows that $\alpha \ge 1$. Note that $\alpha^{\sigma + j} \in \langle \alpha^{n + j} : n \in \ldb 0, \sigma-1 \rdb \rangle$ for all $j \in \nn_0$ and, as a consequence, $\alpha^n \notin \mathcal{A}(\nn_0[\alpha])$ for any $n \ge \sigma$. Now fix $m \in \ldb 0, \sigma - 1 \rdb$, and write $\alpha^m = \sum_{i=0}^k c_i \alpha^i$ for some $c_0, \dots, c_k \in \nn_0$ such that $c_k > 0$. Because $\alpha \ge 1$, we see that $k \le m$. It follows from the minimality of $\sigma$ that $k = m$. As a consequence, $\alpha^m \in \mathcal{A}(\nn_0[\alpha])$. Then we can conclude that $\mathcal{A}(\nn_0[\alpha]) = \{\alpha^n : n \in \ldb 0, \sigma-1 \rdb\}$.
	\smallskip
	
	Finally, we suppose that $\sigma = \infty$. This implies, in particular, that $\alpha \neq 1$. Assume first that $\alpha > 1$. In this case, $\alpha^n$ does not divide $\alpha^m$ in $\nn_0[\alpha]$ for any $n > m$. Therefore $\alpha^m$ is an atom of $\nn_0[\alpha]$ if and only if $\alpha^m$ is not in $\langle \alpha^n : n \in \ldb 0, m-1 \rdb \rangle$. Thus, $\mathcal{A}(\nn_0[\alpha]) = \{\alpha^n : n \in \nn_0\}$. Now assume that $\alpha < 1$. Take $m \in \nn_0$ and suppose that $\alpha^m = \sum_{i=m}^k c_i \alpha^i$ for $c_m, \dots, c_k \in \nn_0$. Observe that $c_m > 0$ as otherwise $1 = \sum_{i=m+1}^k c_i \alpha^{i-m} \notin \mathcal{A}(\nn_0[\alpha])$ and so $\nn_0[\alpha]$ would not be atomic. Then $c_m = 1$, and so $\alpha^m \in \mathcal{A}(\nn_0[\alpha])$. Hence $\mathcal{A}(\nn_0[\alpha]) = \{ \alpha^n : n \in \nn_0 \}$, as desired.
\end{proof}

Observe that we did not use the atomicity of $\nn_0[\alpha]$ to argue that $\mathcal{A}(\nn_0[\alpha]) = \{\alpha^n : n \in \nn_0\}$ in part~(1) of Proposition~\ref{prop:atoms of cyclic algebraic semirings}. Thus, we have that $\nn_0[\alpha]$ is atomic for every transcendental number $\alpha$; indeed, in this case, $\nn_0[\alpha]$ is a free commutative monoid and, therefore, a UFM.

The next corollary follows immediately from Proposition~\ref{prop:atoms of cyclic algebraic semirings}.

\begin{cor} \cite[Corollary~4.3]{CG20}
	For $\alpha \in \rr_{>0}$, the monoid $\nn_0[\alpha]$ is finitely generated if and only if there is an $n \in \nn_0$ such that $\mathcal{A}(\nn_0[\alpha]) = \{\alpha^j : j \in \ldb 0,n \rdb \}$.
\end{cor}

\bigskip


\section{Atomic Positive Monoids Without the ACCP}
\label{sec:ACCP}

It is not hard to argue that every monoid satisfying the ACCP is atomic. However, the converse does not hold in general. Indeed, there are integral domains satisfying the ACCP that are not atomic. The first of such examples was constructed in 1974 by Grams~\cite{aG74}, and further examples were given by Zaks in~\cite{aZ82} and, more recently, by Boynton and Coykendall in~\cite{BC19}. It turns out that there exist valuations of $\nn_0[x]$ that are atomic but do not satisfy the ACCP, and we will construct some of them in this section. Before offering a necessary condition for $\nn_0[x]$ to satisfy the ACCP, we introduce some needed terminology.
\smallskip

For a polynomial $f(x) \in \qq[x]$, we call the set of exponents of the monomial summands of $f(x)$ the \emph{support} of $f(x)$, and we denote it by $\supp \, f(x)$, i.e., $\supp \, f(x) := \{n \in \nn_0 : f^{(n)}(0) \neq 0 \}$, where $f^{(n)}$ denotes the $n$-th formal derivative of $f$. Assume that $\alpha \in \mathbb{C}$ is algebraic over $\qq$, and let $m(x)$ be the  minimal polynomial of $\alpha$. Clearly, there exists a unique $\ell \in \nn$ such that $\ell m(x)$ has content~$1$. Also, there exist unique polynomials $p(x)$ and $q(x)$ in $\nn_0[x]$ such that  $\ell m(x) = p(x) - q(x)$ and $\supp \, p(x) \bigcap \supp \, q(x) = \emptyset$. We say that $(p(x), q(x))$ is the \emph{minimal pair} of $\alpha$.

\begin{prop} \cite[Theorem~4.7]{CG20} \label{prop:ACCP necessary condition}
	Let $\alpha$ be an algebraic number in $(0,1)$ with minimal pair $(p(x),q(x))$. If $\nn_0[\alpha]$ satisfies the ACCP, then $p(x) - x^m q(x)$ is not in $\nn_0[x]$ for any $m \in \nn_0$.
\end{prop}

\begin{proof}
	Assume that the monoid $\nn_0[\alpha]$ satisfies the ACCP. Now suppose, by way of contradiction, that there exists $m \in \nn_0$ with $f(x) := p(x) - x^mq(x) \in \nn_0[x]$. For each $n \in \nn$, we see that
	\[
		q(\alpha)\alpha^{nm} = p(\alpha) \alpha^{nm} = \big( f(\alpha) + \alpha^m q(\alpha) \big) \alpha^{nm} = f(\alpha) \alpha^{nm} + q(\alpha) \alpha^{(n+1)m}.
	\]
	Therefore $\big( q(\alpha)\alpha^{nm} + \nn_0[\alpha] \big)_{n \in \nn}$ is an ascending chain of principal ideals in $\nn_0[\alpha]$. Since $\nn_0[\alpha]$ satisfies the ACCP, such a sequence must eventually stabilize. However, this would imply that $q(\alpha)\alpha^{nm} = \min (q(\alpha)\alpha^{nm} + \nn_0[\alpha]) = \min (q(\alpha)\alpha^{(n+1)m} + \nn_0[\alpha]) = q(\alpha)\alpha^{(n+1)m}$ for some $n \in \nn$, which is clearly a contradiction.
\end{proof}

As a consequence of Proposition~\ref{prop:ACCP necessary condition}, we obtain the following.

\begin{cor} \cite[Theorem~6.2]{GG18}, \cite[Corollary~4.4]{CGG21} \label{cor:atomic rational cyclic semirings without ACCP}
	If $q$ is a rational number in $(0,1)$ such that $\mathsf{n}(q) \ge 2$, then $\nn_0[q]$ is an atomic monoid that does not satisfy the ACCP.
\end{cor}

As we have mentioned before, for each $\alpha \in \rr_{> 0}$, the monoid $\nn_0[\alpha]$ is indeed a semiring. When $\alpha$ is  a transcendental number, $\nn_0[\alpha] \cong \nn_0[x]$ as semirings, and, therefore, a simple degree argument shows that the multiplicative monoid $\nn_0[\alpha]^\bullet$ is atomic. Factorizations of the multiplicative monoid $\nn_0[x]^\bullet$ were studied by Campanini and Facchini in~\cite{CF19}. However, the following question remains unanswered.

\begin{question}\footnote{A version of this question is stated in \cite[Section~3]{BG20} as a conjecture.}
	For which algebraic numbers $\alpha \in \rr_{> 0}$ is the multiplicative monoid $\nn_0[\alpha]^\bullet$ atomic?
\end{question}

We can actually use Corollary~\ref{cor:atomic rational cyclic semirings without ACCP} to construct positive monoids of any prescribed rank that are atomic, but do not satisfy the ACCP. As far as we know, the following result does not appear in the current literature.

\begin{prop} \label{prop:atomic positive monoid without the ACCP}
	For any rank $s \in \nn$, there exists an atomic positive monoid with rank $s$ that does not satisfy the ACCP.
\end{prop}

\begin{proof}
	Fix $s \in \nn$. Since $\rr$ is an infinite-dimensional vector space over $\qq$, we can take $S \subset \rr_{> 0}$ to be a linearly independent set over $\qq$ such that $|S| = s-1$ and $\qq \cap S = \emptyset$. Take then $q \in \qq \cap (0,1)$ with $\mathsf{n}(q) \neq 1$. Because $\nn_0[q]$ is a Puiseux monoid, $\gp(\nn_0[q])$ is an additive subgroup of $\qq$, and so $\rank(\nn_0[q]) = 1$. Now consider the positive monoid $M := \langle \nn_0[q] \cup S \rangle$. It is not hard to see that $M = \nn_0[q] \bigoplus_{s \in S} s\nn_0$. Since $\rank(\nn_0[q]) = 1$, it follows that $\rank(M) = \rank(\nn_0[q]) + s - 1 = s$. Because all direct summands in $\nn_0[q] \bigoplus_{s \in S \setminus \{1\}} s\nn_0$ are atomic, $M$ must be atomic. Consider the sequence of principal ideals $(\mathsf{n}(q)q^n + \nn_0[q])_{n \in \nn_0}$ of $\nn_0[q]$. Since
	\[
		\mathsf{n}(q)q^n = \mathsf{d}(q)q^{n+1} = (\mathsf{d}(q) - \mathsf{n}(q))q^{n+1} + \mathsf{n}(q)q^{n+1},
	\]
	$\mathsf{n}(q)q^{n+1} \mid_{\nn_0[q]} \mathsf{n}(q)q^n$ for every $n \in \nn_0$. Therefore $(\mathsf{n}(q)q^n + \nn_0[q])_{n \in \nn_0}$ is an ascending chain of principal ideals. In addition, it is clear that such a chain of ideals does not stabilize. As a result, $\nn_0[q]$ does not satisfy the ACCP, from which we obtain that $M$ does not satisfy the ACCP.
\end{proof}

Proposition~\ref{prop:atomic positive monoid without the ACCP} allows us to state the following remark in connection to Diagram~\eqref{diag:AAZ's chain for monoids}.

\begin{remark}
	The converse of the implication \textbf{ACCP} $\Rightarrow$ \textbf{atomic} does not hold in the class of positive monoids.
\end{remark}
\smallskip

Our next task will be to construct, for each cardinal number in $\nn$, a class of positive monoids satisfying the ACCP but failing to be BFMs. To do so, we first construct Puiseux monoids that satisfy the ACCP but are not atomic, and then we achieve positive monoids with any prescribed rank by mimicking the technique used in the proof of Proposition~\ref{prop:atomic positive monoid without the ACCP}.

\begin{prop} \cite[Example~2.1]{AAZ90}, \cite[Proposition~4.2.2]{fG21} \label{prop:a class of ACCP monoids}
		For any $s \in \nn_0$, there exists a positive monoid with rank $s$ that satisfies the ACCP but is not a BFM.
\end{prop}

\begin{proof}
	Let $(d_n)_{n \in \nn}$ be a strictly increasing sequence of positive integers with $d_1 \ge 2$ such that $\gcd(d_i, d_j) = 1$ for any distinct $i,j \in \nn$. Consider the monoid $M := \langle 1/d_n : n \in \nn \rangle$. It is not hard to verify that $1/d_j \in \mathcal{A}(M)$ for every $j \in \nn$ and, therefore, $M$ is an atomic monoid with $\mathcal{A}(M) = \{1/d_j : j \in \nn\}$. In addition, we can easily check that for each $q \in M^\bullet$, we can take $n \in \nn_0$ and $c_1, \dots, c_k \in \nn_0$ with $c_k \neq 0$ satisfying 
	\begin{align} \label{eq:decomposition existence}
		q = n + \sum_{i=1}^k c_i \frac{1}{d_i},
	\end{align}
	where $c_i \in \ldb 0,d_i - 1 \rdb$ for each $i \in \ldb 1,k \rdb$. We claim that the decomposition in~(\ref{eq:decomposition existence}) is unique. To argue our claim, take $n' \in \nn_0$ and $c'_1, \dots, c'_m \in \nn_0$ with $c'_m \neq 0$ and $c'_i \in \ldb 0, d_i - 1 \rdb$ for all $i \in \ldb 1, m \rdb$ such that
	\begin{align} \label{eq:decomposition uniqueness}
		n + \sum_{i=1}^k c_i \frac{1}{d_i} = n' + \sum_{i=1}^m c'_i \frac{1}{d_i}.
	\end{align}
	We can complete with zero coefficients if necessary to assume, without loss of generality, that $m = k$. Set $d = d_1 \cdots d_k$, and $n_i := d/d_i \in \nn$ for every $i \in \ldb 1, k \rdb$. Now, for each $j \in \ldb 1,k \rdb$, we can rewrite~\eqref{eq:decomposition uniqueness} as follows:
	\[
		(c_j - c'_j) n_j = (n' - n)d + \sum_{i \neq j} (c'_i - c_i)n_i.
	\]
	Since $d_j \mid d$ and $d_j \mid n_i$ for every $i \in \ldb 1,k \rdb \setminus \{j\}$, the right-hand side of the last equality is divisible by $d_j$. Thus, $(c_j - c'_j) n_j$ is divisible by $d_j$ and because $\gcd(n_j, d_j) = 1$, we see that $d_j \mid (c_j - c'_j)$. This implies that $c'_j = c_j$ for each $j \in \ldb 1, k \rdb$, and so $n' = n$. Hence the decomposition in \eqref{eq:decomposition existence} is unique, as claimed.
	
	With notation as in~\eqref{eq:decomposition existence}, define $N(q) := n$ and $S(q) := \sum_{i=1}^k c_i$. If $q'$ divides $q$ in $M$, then it is clear that $N(q') \le N(q)$. Also, if $q'$ properly divides $q$ in $M$, then the equality $N(q') = N(q)$ guarantees that $S(q') < S(q)$. Putting the last two observations together, we conclude that each sequence $(q_n)_{n \in \nn}$ in $M$ satisfying $q_{n+1} \mid_M q_n$ for every $n \in \nn$ must eventually terminate. Hence the Puiseux monoid $M$ satisfies the ACCP.
	
	Since $M$ is a Puiseux monoid, $\rank(M) = 1$. In addition, we observe that $d_n \in \mathsf{L}_M(1)$ for every $n \in \nn$ because $1 = d_n \frac{1}{d_n}$, whence $M$ is not a BFM. Thus, we have found a rank-one positive monoid that satisfies the ACCP but is not a BFM.
	
	Now, as we did in the proof of Proposition~\ref{prop:atomic positive monoid without the ACCP}, let us take a $\qq$-linearly independent set $S \subset \rr_{> 0}$ such that $|S| = s-1$ and $\qq \cap S = \emptyset$. Then consider the positive monoid $M_s := M \bigoplus_{s \in S} s\nn_0$. As $\rank(M) = 1$, it follows that $\rank(M_s) = s$. In addition, since $M$ satisfies the ACCP, each direct summand of $M_s$ satisfies the ACCP, which implies that $M_s$ also satisfies the ACCP. Since $M$ is a divisor-closed submonoid of $M_s$, the fact that $M$ is not a BFM immediately implies that $M_s$ is not a BFM. Thus, the positive monoid $M_s$ has rank $s$, satisfies the ACCP, but is not a BFM.
\end{proof}

Then we can state the following remark in connection to Diagram~\eqref{diag:AAZ's chain for monoids}.

\begin{remark}
	The converse of the implication \textbf{BFM} $\Rightarrow$ \textbf{ACCP} does not hold in the class of positive monoids.
\end{remark}
\smallskip

\bigskip


\section{The Bounded Factorization Property}

We begin this section providing two equivalent sufficient conditions for a positive monoid to be a BFM.

\begin{prop} \cite[Proposition~4.5]{fG19} \label{prop:BF sufficient condition}
	For a positive monoid $M$ the following statements are equivalent.
	\begin{enumerate}
		\item $\inf M^\bullet > 0$.
		\smallskip
		
		\item $M$ is atomic and $\inf \mathcal{A}(M) > 0$.
	\end{enumerate}
	In addition, any of the above conditions implies that $M$ is a BFM.
\end{prop}

\begin{proof}
	(1) $\Rightarrow$ (2): Since $\inf M^\bullet > 0$, the inclusion $\mathcal{A}(M) \subseteq M^\bullet$ guarantees that $\inf \mathcal{A}(M) > 0$. Let us verify now that $M$ is atomic. Because $\inf M^\bullet > 0$, we can take $\epsilon \in \rr_{> 0}$ satisfying $\epsilon < \inf M^\bullet$. Take now $y \in M^\bullet$ such that $y = b_1 + \cdots + b_n$ for some $b_1, \dots, b_n \in M^\bullet$. Then $y \ge n \min\{b_1, \dots, b_n\} \ge n \epsilon$, which implies that $n \le y/\epsilon$. Then there exists a maximum $m \in \nn$ such that $y = a_1 + \cdots + a_m$ for some $a_1, \dots, a_m \in M^\bullet$. In this case, the maximality of $m$ ensures that $a_1, \dots, a_m \in \mathcal{A}(M)$. As a result, $M$ must be atomic.
	\smallskip
	
	(2) $\Rightarrow$ (1): Take $\epsilon \in \rr_{> 0}$ such that $\epsilon < \inf \mathcal{A}(M)$. For each $r \in M^\bullet$, the fact that $M$ is atomic guarantees the existence of $a \in \mathcal{A}(M)$ dividing $r$ in $M$, and so $r \ge a > \epsilon$. As a result, $\inf M^\bullet > 0$.
	\smallskip
	
	We have seen in the first paragraph that if we take $\epsilon \in \rr_{> 0}$ with $\epsilon < \inf M^\bullet$, then each $y \in M^\bullet$ can be written as the sum of at most $\lfloor y/\epsilon \rfloor$ atoms, and this implies that $\mathsf{L}(y)$ is bounded.  As a consequence, $M$ is a BFM.
\end{proof}

The reverse implication of Proposition~\ref{prop:BF sufficient condition} does not hold. The following example sheds some light upon this observation.

\begin{example} \label{ex:an FF PM having 0 as a limit point}
	Since $\rr$ is an infinite-dimensional vector space over $\qq$, we can take a sequence $(r_n)_{n \in \nn}$ of real numbers whose underlying set is linearly independent over $\qq$. After dividing each $r_n$ by a large enough positive integer $d_n$, one can further assume that $(r_n)_{n \in \nn}$ decreases to zero. Therefore $M = \langle r_n : n \in \nn \rangle$ is a UFM with $\mathcal{A}(M) = \{r_n : n \in \nn\}$. In particular, $M$ is a BFM with $\inf M^\bullet = 0$.
\end{example}

Let us now identify a class of positive monoids that are BFMs but are neither FFMs nor HFMs.

\begin{example} \label{ex:BFM that is neither FFM nor HFM}
	Consider the positive monoid $M := \{0\} \cup \rr_{\ge 1}$. It follows from Proposition~\ref{prop:BF sufficient condition} that~$M$ is a BFM. Note that $\mathcal{A}(M) = [1,2)$. Let us show that $M$ is not an FFM. To do this, note that for each $b \in (2,3]$ the formal sum $(1 + 1/n) + (b - 1 - 1/n)$ is a factorization of  length $2$ in $\mathsf{Z}(b)$ provided that $n \ge \big\lceil \frac{1}{b-2} \big\rceil$. This implies that $|\mathsf{Z}(b)| = \infty$ for all $b \in M_{>2}$. To see that $M$ is not an HFM, it suffices to observe that $3 = 3 \cdot 1 = 2 \cdot \frac{3}{2}$, which implies that $\{2,3\} \subseteq \mathsf{L}(3)$.
\end{example}

In light of Example~\ref{ex:BFM that is neither FFM nor HFM}, we make the following observation.

\begin{remark}
	The converse of the implications \textbf{HFM} $\Rightarrow$ \textbf{BFM} and \textbf{FFM} $\Rightarrow$ \textbf{BFM} do not hold in the class of positive monoids.
\end{remark}
\smallskip

We can generalize the monoid in Example~\ref{ex:BFM that is neither FFM nor HFM} and create two classes of positive monoids that are BFMs whose sets of atoms can be nicely described. These classes of monoids are quite suitable to provide counterexamples, as we just did in Example~\ref{ex:BFM that is neither FFM nor HFM}.

\begin{prop} \cite[Example~4.7]{AG20}, \cite[Proposition~3.14]{BG20} \label{thm:BF sufficient condition}
	Let $r,s \in \rr_{> 0}$ with $r>1$. Then the following statements hold.
	\begin{enumerate}
		\item $M_s = \{0\} \cup \rr_{\ge s}$ is a BFM with $\mathcal{A}(M_s) = [s, 2s)$.
		\smallskip
		
		\item $S_r = \nn_0 \cup \rr_{\ge r}$ is a BFM with $\mathcal{A}(S_r) = \big( \{1\} \cup [r,r+1) \big) \setminus \{ \lceil r \rceil \}$.
	\end{enumerate}
\end{prop}

\begin{proof}
	(1) As $\inf M_s^\bullet = s > 0$, it follows from Proposition~\ref{prop:BF sufficient condition} that $M_s$ is a BFM. In addition, since $2s$ is a lower bound for the set $[s,2s) + [s,2s)$, it follows that $[s,2s) \subseteq \mathcal{A}(M_s)$. Finally, it is clear that $[s,2s)$ generates $M_s$, which implies that $\mathcal{A}(M_s) = [s,2s)$.
	\smallskip
	
	(2) Once again, it follows from Proposition~\ref{prop:BF sufficient condition} that $S_r$ is a BFM. In addition, it is clear that $\rr_{\ge r+1} \subseteq 1 + S_r^\bullet$. As a consequence, $\mathcal{A}(S_r) \subseteq S^\bullet_r \cap \rr_{< r+1} =  \ldb 1, \lceil r \rceil \rdb \cup [r, r+1)$. Since $1 \in \mathcal{A}(S_r)$ and $m \notin \mathcal{A}(S_r)$ for any $m$ in the discrete interval $\ldb 2, \lceil r \rceil \rdb$, we can conclude that $\mathcal{A}(S_r) = \big( \{1\} \cup [r,r+1) \big) \setminus \{ \lceil r \rceil \}$.
\end{proof}

\bigskip


\section{The Finite Factorization Property}

We turn our discussion to the finite factorization property on the class of positive monoids. A positive monoid $M$ is called \emph{increasing} provided that it can be generated by an increasing sequence of positive real numbers. We show that increasing positive monoids are FFMs.

\begin{theorem} \cite[Proposition~3.3]{GG18} \label{thm:increasing PM of Archimedean fields are FF}
	Every increasing positive monoid is an FFM. In addition, if $(r_n)_{n \in \nn}$ is an increasing sequence of positive real numbers generating a positive monoid $M$, then $\mathcal{A}(M) = \{r_n : r_n \notin \langle r_1, \dots, r_{n-1} \rangle\}$.
\end{theorem}

\begin{proof}
	It is clear that $M$ is atomic; indeed, it follows from Proposition~\ref{prop:BF sufficient condition} that $M$ is a BFM. Let us suppose for the sake of contradiction that $M$ fails to be an FFM. Because $M$ is not an FFM, the set
	\[
		X := \{r \in M : \, |\mathsf{Z}(r)| = \infty\}
	\]
	is not empty. Set $s = \inf X$ and note that $s$ is positive. Since $M$ is increasing, it follows that  $m := \inf M^\bullet \in M$. Take $\epsilon \in (0,m)$ and then $r \in X$ with $s \le r < s + \epsilon$. Observe that every $a \in \mathcal{A}(M)$ appears in only finitely many factorizations of $r$. Because  $|\mathsf{L}(r)| < \infty$, we can choose $\ell \in \mathsf{L}(r)$ so that $Z_\ell := \{z \in \mathsf{Z}(r) : \, |z| = \ell\}$ has infinite size. Let $z = a_1 \cdots a_\ell \in Z_\ell$ for some atoms $a_1, \dots, a_\ell$ of $M$ with $a_1 \le \cdots \le a_\ell$. As each of the atoms appears in only finitely many factorizations of the infinite set $Z_\ell$, we can take $z' = a'_1 \cdots a'_\ell \in Z_\ell$ for some atoms $a'_1, \dots, a'_\ell$ of $M$ satisfying $a_\ell < \min\{a'_1, \dots, a'_\ell\}$. Accordingly, we obtain
	\[
		a_1 + \cdots + a_\ell \le \ell a_\ell < a'_1 + \cdots + a'_\ell,
	\]
	which contradicts that both $a_1 \cdots a_\ell$ and $a'_1 \cdots a'_\ell$ are factorizations of the same element, namely, $r$. Hence $M$ is an FFM. 
	\smallskip
	
	In order to argue the second statement, let $A$ denote the set $\{r_n : r_n \notin \langle r_1, \dots, r_{n-1} \rangle\}$. It follows immediately that $A = \mathcal{A}(M)$ if $M$ is finitely generated, in which case, $M$ is an FFM by \cite[Corollary~3.7]{AG20}. We assume, therefore, that $|A| = \infty$. Let $(a_n)_{n \in \nn}$ be a strictly increasing sequence with underlying set $A$. Because $a_1$ is the minimum of $M^\bullet$, it must be an atom. In addition, as $(a_n)_{n \in \nn}$ is strictly increasing, for each $n \ge 2$ the fact that $a_n \notin \langle a_1,\dots, a_{n-1} \rangle$ guarantees that $a_n \in \mathcal{A}(M)$. As a result, $\mathcal{A}(M) = A$, which concludes the proof.
\end{proof}

There are positive monoids that are FFM but are not increasing, that is, the converse of Theorem~\ref{thm:increasing PM of Archimedean fields are FF} does not hold in general.

\begin{example}
	Consider the monoid $M = \langle r_n : n \in \nn \rangle$ constructed in Example~\ref{ex:an FF PM having 0 as a limit point}, where $(r_n)_{n \in \nn}$ is a sequence of real numbers that strictly decreases to zero and whose terms are linearly independent over $\qq$. Since $M$ is a UFM, it is clearly an FFM. However, the fact that $0$ is a limit point of $M^\bullet$ guarantees that $M$ cannot be generated by an increasing sequence of real numbers. Hence $M$ is not an increasing positive monoid.
\end{example}

\medskip
For every $n \in \nn$, we can take elements $r_1, \dots, r_n \in \rr_{> 0}$ that are linearly independent over $\qq$. Consider the positive monoid $M_n := \langle r_1, \dots, r_n \rangle$. It is clear that $M_n$ is a UFM of rank $n$. In the same way, we can create (and have created in previous examples) positive monoids of infinite rank. Since every UFM is an FFM, we have finite factorization positive monoids of any rank. It turns out that just inside the class of positive monoids $\{\nn_0[\alpha] : \alpha \in \rr_{>0}\}$ discussed in Section~\ref{sec:atomicity}, there are FFMs of any rank that are not UFMs. 

\begin{prop} \cite[Theorem~5.4]{CG20} \label{prop:UFM characterization}
	For $\alpha \in \rr_{>0}$, the following statements hold.
	\begin{enumerate}
		
		\item If $\alpha$ is transcendental, then $\nn_0[\alpha]$ is UFM of infinite rank.
		\smallskip
		
		\item If $\alpha > 1$, then $\nn_0[\alpha]$ is an FFM.
		\smallskip
		
		\item If $\alpha$ is algebraic with minimal polynomial $m(x)$, then $\nn_0[\alpha]$ is a UFM if and only if $\deg m(x) = |\mathcal{A}(\nn_0[\alpha])|$.
		
	\end{enumerate}
\end{prop}

\begin{proof}
	(1) Because $\alpha$ is transcendental, there is no nonzero polynomial in $\qq[x]$ having $\alpha$ as a root and, therefore, the set $\{\alpha^n : n \in \nn_0\}$ is linearly independent over $\qq$. Hence $\nn_0[\alpha]$ is a UFM.
	\smallskip
	
	(2) Since $\alpha > 1$, we see that $\alpha^n < \alpha^{n+1}$ for every $n \in \nn_0$. As a result, $(\alpha^n)_{n \in \nn_0}$ is an increasing sequence generating $\nn_0[\alpha]$. Hence $\nn_0[\alpha]$ is an increasing positive monoid, and it follows from Theorem~\ref{thm:increasing PM of Archimedean fields are FF} that it is an FFM.
	\smallskip

	(3) For the direct implication, suppose that $\nn_0[\alpha]$ is a UFM (and so an HFM). If $\alpha \in \qq$, then $\deg m(x) = 1$, while it follows from \cite[Proposition~4.2]{fG20} that $\nn_0[\alpha]$ is isomorphic to the additive monoid $\nn_0$. So in this case, the equalities $\deg m(x) = 1 = |\mathcal{A}(\nn_0[\alpha])|$ hold. We assume, therefore, that $\alpha \notin \qq$, that is, $\deg m(x) > 1$. Now set
	\[
		\sigma = \min \{n \in \nn : \alpha^n \in \langle \alpha^j : j \in \ldb 0, n-1 \rdb \rangle \}.
	\]
	As $\nn_0[\alpha]$ is atomic, $\mathcal{A}(\nn_0[\alpha]) = \{\alpha^j : j \in \ldb 0, \sigma-1 \rdb\}$ by Theorem~\ref{thm:atomic characterization}. Because $m(x)$ divides any polynomial in $\qq[x]$ having $\alpha$ as a root, we obtain that $\deg m(x) \le \sigma$. Suppose for the sake of contradiction that $\deg m(x) < \sigma$. Now take $d \in \nn$ with $dm(x) \in \zz[x]$ and write $dm(x) = p(x) - q(x)$ for polynomials $p(x)$ and $q(x)$ in $\nn_0[x]$. Since $m(\alpha) = 0$, both $p(\alpha)$ and $q(\alpha)$ induce factorizations in $\nn_0[\alpha]$ of the same element. Since $\nn_0[\alpha]$ is a UFM, we see that $m(x) = 1/d(p(x) - q(x)) = 0$, which is a contradiction. Thus, $\deg m(x) = \sigma = |\mathcal{A}(\nn_0[\alpha])|$, as desired.

	For the reverse implication, suppose that $\deg m(x) = |\mathcal{A}(\nn_0[\alpha])|$. As the set $\mathcal{A}(\nn_0[\alpha])$ is not empty, $\nn_0[\alpha]$ is atomic by virtue of Theorem~\ref{thm:atomic characterization}. In addition,
	\[
		\mathcal{A}(\nn_0[\alpha]) = \{\alpha^j : j \in \ldb 0,d-1 \rdb\},
	\]
	where $d$ is the degree of $m_\alpha(x)$. Then $\alpha^d = \sum_{i=0}^{d-1} c_i \alpha^i$ for some $c_0, \dots, c_{d-1} \in \nn_0$, which implies that $m(x) = x^d - \sum_{i=0}^{d-1} c_i x^i$. Let $\sigma$ be as in the previous paragraph.
	Now for any two factorizations $z_1, z_2 \in \mathsf{Z}(\nn_0[\alpha])$ of the same element in $\nn_0[\alpha]$, we have that $\max \{ \deg z_1(x), \deg z_2(x) \} < d$ and $z_1(\alpha) = z_2(\alpha)$. This implies that $m(x)$ divides the polynomial $z_1(x) - z_2(x)$, which has degree strictly less than $m(x)$. As a result, $z_1(x) = z_2(x)$, which implies that $z_1 = z_2$. As a consequence, $\nn_0[\alpha]$ is a UFM.
\end{proof}

Let us show now that, for any rank $n$, there is a positive monoid of rank $n$ that is an FFM but not a UFM. As far as we know, the following result does not appear in the current literature.

\begin{prop}
	For every $n \in \nn$, there exists an algebraic element $\alpha \in \rr_{> 0}$ such that $\nn_0[\alpha]$ is a rank-$n$ FFM that is not a UFM.
\end{prop}

\begin{proof}
	For $n=1$, we can take $M = \langle q^n : n \in \nn_0 \rangle$, where $q \in \qq_{> 1}$. It is clear that $M$ has rank $1$. Since $M$ is generated by the increasing sequence $(q^n)_{n \in \nn_0}$, Theorem~\ref{thm:increasing PM of Archimedean fields are FF} ensures that $M$ is an FFM. Also, it follows from \cite[Proposition~4.2]{fG20} that $M$ is not a UFM. 
	
	For $n \ge 2$, consider the polynomial $m(x) = x^n - 4x + 2 \in \zz[x]$. Since $m(1) = -1$ and $m(4) > 0$, the polynomial $m(x)$ has a root $\alpha$ in the interval $(1,4)$. It follows from Eisenstein's Criterion at the prime ideal $2\zz$ that $m(x)$ is irreducible. As a result, $m(x)$ is the minimal polynomial of $\alpha$. Consider now the monoid $M = \nn_0[\alpha]$. It follows from~\cite[Proposition~3.2]{CG19} that the rank of $M$ equals the degree of $m(x)$, that is, $\rank \, M  = n$. Because $\alpha > 1$, the monoid $M$ is an FFM by virtue of Theorem~\ref{thm:increasing PM of Archimedean fields are FF}. Finally, let us show that $M$ is not a UFM. Suppose, by way of contradiction, that $\deg m(x) = |\mathcal{A}(M)|$. In this case, it follows from Proposition~\ref{prop:atoms of cyclic algebraic semirings} that
	\[
		\mathcal{A}(M) = \{\alpha^j : j \in \ldb 0,n-1 \rdb\}.
	\]
	Then $\alpha^n \in \langle \alpha^j : j \in \ldb 0, n-1 \rdb \rangle$, and so we can take $c_0, \dots, c_{n-1} \in \nn_0$ such that $\alpha^n = \sum_{i=0}^{n-1} c_i \alpha^i$. Now the fact that $\alpha$ is a root of the polynomial $f(x) = x^n - \sum_{i=0}^{n-1} c_i x^i$, which is monic of degree $\deg m(x)$, implies that $f(x) = m(x)$. However, in this case one finds that $c_0 = -f(0) = -m(0) = -2$, a contradiction. As a consequence, $\deg m(x) \neq |\mathcal{A}(M)|$, and so Proposition~\ref{prop:UFM characterization} guarantees that $M$ is not a UFM, which concludes our proof.
\end{proof}

Let us record the following remark in connection to Diagram~\eqref{diag:AAZ's chain for monoids}.

\begin{remark}
	The converse of the implication \textbf{UFM} $\Rightarrow$ \textbf{FFM} does not hold in the class of positive monoids.
\end{remark}

For the sake of completeness, we conclude with an example of a positive monoid that is an HFM but not a UFM; this is \cite[Example~7.2]{BCG21}.

\begin{example} \label{ex:HFM not UFM}
	Take $n \in \nn$ and consider the positive monoid $M_n = \big\langle \pi, n, \frac{\pi + n}{2} \big\rangle$. One can easily show that $\mathcal{A}(M_n) = \big\{ \pi, n, \frac{\pi + n}{2} \big\}$. As $\pi + n$ and $2 \frac{\pi + n}{2}$ are distinct factorizations of $\pi + n$, we see that $M_n$ is not a UFM. Now take $c_1, c_2, c_3 \in \zz$ such that
	\[
		c_1 \pi + c_2 n +c_3 \frac{\pi + n}2 = 0.
	\]
	Since $\pi$ is irrational, $c_1 + c_3/2 = c_2 + c_3/2 = 0$ and, therefore, $c_1 + c_2 + c_3 = 0$. Thus, the positive monoid $M_n$ is an HFM.
\end{example}

We conclude with the following remark in connection to Diagram~\eqref{diag:AAZ's chain for monoids}.

\begin{remark}
	The converse of the implication \textbf{UFM} $\Rightarrow$ \textbf{HFM} does not hold in the class of positive monoids.
\end{remark}

\bigskip


\section*{Acknowledgments}

The authors would like to thank Felix Gotti for helpful conversations during the preparation of this paper.

\bigskip


\end{document}